\documentclass[a4paper]{article}
\usepackage[english]{babel}
\usepackage[utf8x]{inputenc}
\usepackage[T1]{fontenc}

\usepackage[a4paper,top=3cm,bottom=2cm,left=3cm,right=3cm,marginparwidth=1.75cm]{geometry}

\usepackage{graphicx, amsmath, amssymb, amsthm, amscd}
\usepackage{color}
\usepackage{epsf}
\usepackage{enumerate}

\usepackage{amsmath}
\usepackage{graphicx}
\usepackage[colorinlistoftodos]{todonotes}
\usepackage[colorlinks=true, allcolors=blue]{hyperref}

\newtheorem{theorem}{Theorem}[section]

\newtheorem{corollary}[theorem]{Corollary}

\newtheorem{remark}[theorem]{Remark}

\title{On minimal manifolds}
\begin{document}
\maketitle

\begin{center}Jan P. Boro\'nski\footnote{AGH University of Science and Technology, Faculty of Applied Mathematics, al. Mickiewicza 30, 30-059 Krak\'ow, Poland -- and -- National Supercomputing Centre IT4Innovations, Division of the University of Ostrava, Institute for Research and Applications of Fuzzy Modeling, 30. dubna 22, 70103 Ostrava, Czech Republic, e-mail: boronski@agh.edu.pl} and George Kozlowski\footnote{Auburn University, Department of Mathematics and Statistics, e-mail: kozloga@auburn.edu}
\end{center}
\begin{abstract}
Let $M$ be a compact manifold of dimension at least 2. If $M$ admits a minimal homeomorphism then $M$ admits a minimal noninvertible map. 
\end{abstract}

\section{Introduction}
Fathi and Herman \cite{FH} showed that every compact connected manifold which admits a smooth, locally free effective action of the circle group has a smooth diffeomorphism, isotopic to the identity, which is minimal, and so all odd dimensional spheres admit a minimal diffeomorphism (see also \cite{Katok}). Every manifold that admits a minimal flow admits also a minimal homeomorphism. The converse does not hold even in dimension 2, as manifested on the Klein bottle \cite{Parry}. Less is known about the existence of minimal noninvertible maps. The circle does not admit such maps, despite admitting minimal rotations \cite{AK}. 
Both the 2-torus and Klein bottle admit minimal noninvertible maps \cite{KST}, \cite{Sotola}, but in dimension at least $3$ it has not been known which manifolds admit minimal noninvertible maps. By application of a result of B\'eguin, Crovisier and Le Roux \cite{Cro}, here we show in Theorem \ref{main} that any compact manifold of dimension at least 2 admits a minimal noninvertible map if it admits a minimal homeomorphism.  
\section{Preliminaries}
A self map $f:X\to X$ of a compact metric space is said to be \textit{minimal} if for ever point $x\in X$ the forward orbit $\{f^n(x):n\in\mathbb{N}\}$ is dense in $X$. 
\subsection{A Denjoy-type result}
The following result was obtained in \cite{Cro}, p.304, as a corollary of a far-reaching generalization of the Denjoy-Rees technique. It will be crucial to the proof of Theorem \ref{main}.
\begin{theorem}\label{Cro}(B\'eguin,Crovisier\&Le Roux)
Let $R$ be a homeomorphism on a compact manifold $M$, and $x$ a point of $M$ which is not periodic under $R$. Consider a compact subset $D$ of $M$ which can be written as the intersection of a strictly decreasing sequence $\{D_n:n\in\mathbb{N}\}$ of tamely embedded topological closed balls. Then there exist a homeomorphism $f: M\to M$ and a continuous onto map $\phi : M \to M $ such that $\phi \circ f = R \circ \phi$, and such that
\begin{itemize} 
\item $\phi^{−1}(x)=D$;
\item $\phi^{-1}(y)$ is a single point if $y$ does not belong to the $R$-orbit of $x$.
\end{itemize}
\end{theorem}
\begin{remark}\label{null}
It is implicitly mentioned in the proof of Proposition C.1. in \cite{Cro}, p.304, that for the homeomorphism $f$ in Theorem \ref{main} we have  $\lim_{n\to\infty}\operatorname{diam}(f^n(D))=0$ (see also $\mathbf{B_3}$ and Proposition 3.1 on p. 270 of \cite{Cro}). The fact that if R is minimal and D has empty interior, then f is minimal is explicitly stated in Remarks following Proposition C.1. in \cite{Cro}, p.304.
\end{remark}
\subsection{Decomposition of Manifolds}
The theorem of this section derives from a result of Bing \cite{Bing} for upper semicontinous decompositions of $E^3$ into points and a countable family of flat arcs.
As stated by Daverman, it was a part of folklore that 
Bing's proof generalizes to $E^n$ for all dimensions $n$ greater then 2
(see e.g. \cite{Daverman} Chapter II, Theorem 2 of Section 5, p.23), and the extension to manifolds is stated as Exercise 1 of Section 8, p.61. For surfaces a more general result was inspired by 
Moore's Decomposition Theorem \cite{Moore} for the sphere and proved in general in \cite{Roberts} (see Theorem 1).  A modern treatment of this result also appears in \cite{Daverman}, (see Chapter IV, Theorem 1 of Section 25, p.187). 

For completeness a proof  is provided of the exercise by a method which applies to more general situations.

It will be useful to recall some definitions which can be found in 
\cite{Daverman}.
An arc $A$ in a manifold $M$ is called \textit{essentially flat} (p.~61) if there exist a neighborhood $U$ of $A$ and a homeomorphism $\psi$ of $U$ onto $E^n$ such that $\psi(A)$ is flat.
A surjective map $f \colon X \to Y$ is a \emph{near-homeomorphism}, if for any open cover $\mathcal{V}$ of the $Y$ there is a homeomorphism  $h \colon X \to Y$ 
such that for every $x$ in $X$ there is $V$ in $\mathcal{V}$ which contains $f(x)$ and $h(x)$ (p.~27).  If $\mathcal{G}$ is an upper semicontinuous of a space $X$, a \textit{realization} of $\mathcal{G}$ is a closed map $f \colon X \to Y$ such that $\mathcal{G}$  is precisely the set of all point-inverses $f^{-1}(y)$ ($y \in Y$) (see Chapter 1, Theorem 5, p.~11); in this case $X/\mathcal{G}$ is homeomorphic to $Y$.  A set $P$  is called \textit{$\mathcal{G}$-saturated} if $P$ is the union of elements of $\mathcal{G}$.

\begin{theorem}\label{Daverman}\cite{Daverman}
Suppose $\mathcal{G}$ is a upper semicontinuous decomposition of a manifold $M$ into points and a countable family of essentially flat arcs. Then $M/\mathcal{G}$ is homeomorphic to $M$. 
\end{theorem}
\begin{proof} The following consequences of the generalization of Bing's theorem \cite{Bing} will be used repeatedly in the proof.  The first two assertions follow immediately from \cite{Daverman}; the third assertion follows from the second, the last from the first.

\textbf{Preamble}. 
Suppose $\mathcal{G}$ is a upper semicontinuous decomposition of a manifold $M$ into points and a countable family of essentially flat arcs for which there is a compact subset $A$ of $M$ which has a neighborhood homeomorphic to $E^n$ and which contains all the nondegenerate elements of $\mathcal{G}$. Then 
\begin{enumerate}
\item there is a map $f \colon M \to M$ which realizes $\mathcal{G}$; 
\item for any neighborhood $V$ of $A$ and any $\varepsilon >0$ there is a homeomorphism $h \colon M \to M$ such that for all points $p$ in $V$ the distance between $f(p)$ and $h(p)$ is less than  $\varepsilon$ and $h(p)=p$ for all points $p$ of $M$ not in $V$;
\item if $B$ is a compact subset of $M$ which has a neighborhood homeomorphic to $E^n$, the same is true of $f(B)$;
\item If $J$ is an essentially flat arc which contains no point of any nondegenerate element of $\mathcal{G}$, then $f(J)$ is also an essentially flat arc.
\end{enumerate}

With the preamble concluded the proof of the theorem begins with choices of 
locally finite covers $\{ V_i \}_{1 \le i < \omega}$ and  $\{ A_i \}_{1 \le i < \omega}$ of $M$ having the following properties for every index $i$
\begin{enumerate}
  \item  $A_i$ is compact and $\mathcal{G}$-saturated.
  \item  $A_i$ has a neighborhood homeomorphic to $E^n$. 
  \item  $V_i$ has compact closure and is $\mathcal{G}$-saturated.
  \item there is a finite subset $H$ of indices such that  $V_i$ contains a point of $V_j$ only if $j \in H$.
\end{enumerate}

Let $\mathcal{G}_1$ be the upper semicontinuous decomposition of $M$ whose  elements are those members $g \in \mathcal{G}$ for which $g \subset A_1$ and the singletons of points in $M$ which are not in $A_1$.  
It follows from the Preamble that there is a near homeomorphism $f_1 \colon M \to M$ which satisfies its conclusions (affixing the subscript 1).

For an induction suppose near homeomorphisms $f_i \colon M \to M$ have been obtained for $1 \le i < r$ for some $r$ with $1 \le r < \omega$.  Let $\mathcal{G}_r$ be the decomposition whose nondegenerate elements are members of $\mathcal{G}$ of the form $f_{r-1} \cdots f_1 (g)$ with $g \subset A_r$.  The set $f_{r-1} \cdots f_1 A_r$ has a neighborhood homeomorphic to $E^n$, and the nondegenerate elements of $\mathcal{G}_r$ are a countably family of essentially flat arcs.  Hence there is a near homeomorphism $f_r$ realizing $\mathcal{G}_r$ which satisfies the conclusions of the Preamble (affixing the subscript r).  

Inductively, a sequence of maps $f_1, f_2, \dots$ results such that the composition $f_r \cdots f_1$ is a near homeomorphism which realizes the decomposition consisting of those members of $\mathcal{G}$ which lie in $A_1 \cup \cdots \cup A_r$ and the singletons of all points of $M$ which do not belong to $A_1 \cup \cdots \cup A_r$.

For each $i$ there is $s$ such that $f_r \cdots f_1(p) = f_s \cdots f_1(p)$ for all $p$ in $V_i$ and all $r >s$.  It follows that the map $f \colon M \to M$ defined by $f(p)= \lim_{r < \omega} f_r \cdots f_1(p)$ is a map which realizes $\mathcal{G}$, and consequently $M$ is homeomorphic to $M/\mathcal{G}$.
\end{proof}
\section{Main Result}
\begin{theorem}\label{main} 
Let $M$ be a compact manifold of dimension at least 2. If $M$ admits a minimal homeomorphism then $M$ admits a minimal noninvertible map. 
\end{theorem}
\begin{proof}
Let $h:M\to M$ be a minimal homeomorphism, and $A_0$ be an essentially flat arc in $M$. By Theorem \ref{Cro} there exist a homeomorphism $f: M\to M$ and a continuous onto map $\phi : M \to M $ such that $\phi \circ f = h \circ \phi$, and such that $\phi^{−1}(x)=A_0$, and $\phi^{-1}(y)$ is a single point if $y$ does not belong to the $h$-orbit of $x$. In particular, $\phi^{−1}(h^n(x))=A_n$ is an essentially flat arc for every $n\in \mathbb{Z}$. Note that $f$ is minimal, as each $A_n$ is nowhere dense; i.e. $f$ is an almost one-to-one extension of a minimal homeomorphism. Let $\mathcal{G}$ be the decomposition of $M$ into arcs $\{A_n:n>0\}$ and points that are not in $\bigcup_{n=1}^\infty A_n$. Since by Remark \ref{null} the family $\{A_n:n>0\}$ forms a null sequence, the decomposition $\mathcal{G}$ is upper semicontinuous (see e.g. Proposition 3, p. 14 in \cite{Daverman}). Then by Proposition \ref{Daverman} the quotient space $M/\mathcal{G}$ is homeomorphic to $M$, and $f$ induces a map $\hat{f}$ such that $\tau \circ \hat{f} = f \circ \tau$, where $\tau:M\to M/\mathcal{G}$ is the decomposition map. Note that $\hat{f}$ is minimal, as a factor of a minimal map, and $\hat{f}$ is noninvertible at $\tau(A_0)$; i.e. $\hat{f}(\tau(A_0))$ is a point; cf. \cite{KST}. This completes the proof.
\end{proof}
Theorem \ref{main} combined with the result of Fathi and Herman gives the following corollary, which can be considered as a counterexample to a ''noninvertible Gottschalk Conjecture". 
\begin{corollary}
All odd dimensional spheres admit a minimal noninvertible map.  
\end{corollary}
Other recent developments, concerning admissability of minimal noninvertible maps on topological spaces, can be found in \cite{BCO}, \cite{BCFK}, \cite{BKLO} and \cite{SS}. 
\section{Acknowledgements}
The first author is grateful to Krystyna Kuperberg for valuable comments. This work was supported by National Science Centre, Poland (NCN), grant no. 2015/19/D/ST1/01184.

\bibliographystyle{alpha}

\end{document}